\documentclass[a4paper,11pt] {amsart}
\usepackage[utf8]{inputenc}
\usepackage[english]{babel}
\usepackage{hyperref}
\usepackage{amsmath}
\usepackage{amsthm}
\usepackage{amsfonts}
\usepackage{amssymb}
\usepackage{graphicx}
\usepackage{a4wide}
\usepackage{hyphenat}
\usepackage{latexsym}
\usepackage{tcolorbox}
\usepackage[all]{xy}
\usepackage{enumerate}
\usepackage{mathrsfs}
\usepackage{xspace, psfrag, setspace, supertabular}
\usepackage{color}
\newcommand{\comm}[1]{}
\reversemarginpar
\numberwithin{equation}{section}
\usepackage{latexsym}
\usepackage{tcolorbox}
\usepackage[all]{xy}
\numberwithin{equation}{section}
\usepackage{mathrsfs}

\theoremstyle{plain}
\newtheorem{thm}{Theorem}[section]
\newtheorem{lemma}[thm]{Lemma}
\newtheorem{prop}[thm]{Proposition}
\newtheorem{cor}[thm]{Corollary}

\newtheorem*{claims}{Claim}

\theoremstyle{definition}
\newtheorem{defi}[thm]{Definition}
\newtheorem{rem}[thm]{Remark}
\newtheorem{rems}[thm]{Remarks}

\newtheorem{nota}[thm]{Notation}

\newtheorem*{step1}{Step 1}
\newtheorem*{step2}{Step 2}
\newtheorem*{step3}{Step 3}

\newcommand{\N}{\mathbb{N}}
\newcommand{\Z}{\mathbb{Z}}

\newcommand{\R}{\mathbb{R}}
\newcommand{\C}{\mathbb{C}}
\newcommand{\T}{\mathbb{T}}
\newcommand{\D}{\mathbb{D}}

\newcommand{\CA}{\mathcal{A}}

\newcommand{\CE}{\mathcal{E}}

\newcommand{\CH}{\mathcal{H}}

\newcommand{\FD}{\mathfrak{D}}

\newcommand{\al}{\alpha}
\newcommand{\be}{\beta}
\newcommand{\ga}{\gamma}
\newcommand{\de}{\delta}

\newcommand{\la}{\lambda}
\newcommand{\ep}{\varepsilon}
\newcommand{\si}{\sigma}
\newcommand{\La}{\Lambda}
\newcommand{\ka}{\varkappa}

\newcommand{\sm}{\setminus}
\renewcommand{\ss}{\subset}
\newcommand{\tn}[1]{\textnormal{#1}}
\newcommand{\ol}[1]{\overline{#1}}

\newcommand{\abs}[1]{| #1 |}
\newcommand{\Abs}[1]{\left| #1 \right|}
\newcommand{\norm}[1]{\left\lVert #1 \right\rVert}
\newcommand{\Norm}[1]{{\vert\kern-0.25ex\vert\kern-0.25ex\vert #1 \vert\kern-0.25ex\vert\kern-0.25ex\vert}}
\newcommand{\ip}[2]{\left\langle #1 , #2 \right\rangle}
\newcommand{\wt}[1]{\widetilde{#1}}
\newcommand{\wh}[1]{\widehat{#1}}

\newcommand{\AW}{A_{W,\R}}
\newcommand{\AWC}{A_{W}}

\newcommand{\op}{\wh{T}}
\newcommand{\func}{\gamma}

\newcommand{\weak}[1]{\operatorname{Adm}_{#1}^{w}}

\newcommand{\adm}[1]{\operatorname{Adm}_{#1}}

\date{\today}
\title{Functional models up to similarity and $a$-contractions}

\author{Luciano Abadias}
\address{Luciano Abadias \newline
Departamento de Matem\'aticas, \newline
Instituto Universitario de
Matem\'aticas y Aplicaciones, \newline
Universidad de Zaragoza, \newline
50009 Zaragoza, Spain}
\email{labadias@unizar.es}

\author{Glenier Bello}
\address{Glenier Bello\newline
Departamento de Matem\'aticas,\newline
Universidad Aut\'onoma de Madrid,\newline
Cantoblanco, 28049 Madrid, Spain,\newline
and Instituto de Ciencias Matem\'aticas (CSIC-UAM-UC3M-UCM)}
\email{glenier.bello@uam.es}

\author{Dmitry Yakubovich}
\address{D. V. Yakubovich\newline
Departamento de Matem\'aticas,\newline
Universidad Aut\'onoma de Madrid,\newline
Cantoblanco, 28049 Madrid, Spain\newline
and Instituto de Ciencias Matem\'aticas (CSIC-UAM-UC3M-UCM)}
\email {dmitry.yakubovich@uam.es}

\usepackage{fancyhdr}
\begin{document}

\begin{abstract}
We study the generalization of $m$-isometries and $m$-contractions
(for positive integers $m$) to what we call $a$-isometries and $a$-contractions
for positive real numbers $a$.
We show that operators satisfying certain inequality in hereditary form
are similar to $a$-contractions. This improvement of \cite[Theorem~I]{BY18}
is based on some Banach algebras techniques.
The study of $a$-isometries and $a$-contractions relies on
properties of finite differences.
\end{abstract}


\maketitle

\section{Introduction}
Let $\al(t)$ be a function representable by power series
$\sum_{n=0}^{\infty} \al_n t^n$
in the unit disc $\D := \{ \abs{t} < 1 \}$.
If $T \in L(H)$ is a bounded linear operator on a separable Hilbert space $H$,
we put
\begin{equation}
\label{ser-alpha}
\al(T^*,T) := \sum_{n=0}^{\infty} \al_n T^{*n} T^n
\end{equation}
whenever this series converges in the strong operator
topology SOT in $L(H)$.

Since Sz.-Nagy and Foia\c{s} developed their spectral theory for contractions
(see \cite{NFBK10}) based on the construction of a functional model,
an intensive research has been done on obtaining a functional model for operators
$T$ such that $\al(T^*,T) \ge 0$ for distinct types of functions $\al$.
Note that when $\al(t) = 1-t$ the operator
inequality $\al(T^*,T) \ge 0$ means that $T$ is a contraction.
The usual approach, which goes back to Agler (see \cite{Agl82}), 
is based on the assumption that the function
$k(t) := 1/\al(t)$ defines a reproducing kernel Hilbert space.
In particular, it is assumed that $\al$ does not vanish on the unit disc $\D$.
We refer the reader to the introduction of our recent paper
\cite{ABY19_1} and the references therein for more details. 

In \cite{BY18}, the last two authors consider functions $\al$ in the
Wiener algebra $A_W$ of analytic functions in the unit disc $\D$
with summable sequence of Taylor coefficients.
In that paper, the so called \emph{admissible functions} $\al$
have the form $\al(t) = (1-t) \wt{\al}(t)$, where $\wt{\al}$ belongs to
$A_W$, has real Taylor coefficients, and is positive on the interval $[0,1]$.
Therefore the admissible functions $\al$ may have zeroes in $\D \sm [0,1)$.
In \cite[Theorem~I]{BY18} it was proved that 
whenever $\al(T^*,T) \ge 0$, $T$ is similar to a Hilbert space contraction.
In Theorem~\ref{thm quitar alpha tilde} we generalize this result.

This paper should be seen as a second part of \cite{ABY19_1},
where we focused on unitarily equivalent models for operators $T$
satisfying $\al(T^*,T) \ge 0$ for certain functions $\al$.
We tried to make our exposition independent of 	
\cite{ABY19_1}. 

We need to introduce some notation.
Given a function $\al$ in the Wiener algebra $A_W$
and an operator $T$ in $L(H)$ with spectrum $\si(T) \ss \ol{\D}$,
if the series
$\sum \abs{\al_n} T^{*n} T^n$
converges in SOT, then we say that
$\al \in \CA_T$. According to the notation of \cite{ABY19_1}, this means precisely
that $T \in \weak{\al}$.
Depending on whether we want to consider fixed the operator $T$ or
the function $\al$, we will use one notation or the other.
In the same way,
if the series
$\sum \abs{\al_n} T^{*n} T^n$
converges in the uniform operator topology on $L(H)$, then we say that
$\al \in \CA_T^0$, or equivalently $T \in \adm{\al}$.
Hence,
\[
\CA_T^0 \ss \CA_T \ss A_W,
\qquad \tn{ and } \qquad
\adm{\al} \ss \weak{\al} \ss L(H).
\]
We will see later that $\CA_T$ is a Banach algebra
(the norm is given by \eqref{eq defi norm in AT})
and $\CA_T^0$ is
a separable closed subalgebra of $\CA_T$.

We denote by $\CA_{T,\R}^0$, $\CA_{T,\R}$, and $\AW$ the subsets of functions
in $\CA_{T}^0$, $\CA_{T}$, and $A_W$,  respectively, whose Taylor coefficients are real.
It turns out (see Proposition~\ref{prop fnTBT SOT})
that if $\al \in \CA_{T,\R}$, then the operator
$\al(T^*,T)$ is well defined (that is, the series \eqref{ser-alpha}
converges in SOT).

\begin{thm}\label{thm quitar alpha tilde}
Let $\op$ be an operator in $L(H)$ with spectrum $\si(\op) \ss \ol{\D}$,
and let $\ga$ be a function of the form
\[
\ga(t) = \al(t) \wt{\ga}(t),
\]
where $\al \in \CA_{\op,\R}$, and $\wt{\ga} \in \CA_{\op,\R}^0$ is positive on the
interval $[0,1]$. If $\ga(\op^*,\op) \ge 0$,
then $\op$ is similar to an operator $T \in L(H)$ 
such that $\al(T^*,T) \ge 0$.
\end{thm}

In other words, this theorem permits one to leave out (in terms of similarity) 
the factor $\wt{\func}$.
Therefore, the question of obtaining a model for $\op$ up to similarity
reduces to obtaining a model for $T$ (as above).
An important case is $\al(t) := (1-t)^a$ for $a>0$.
If $(1-t)^a(T^*,T) \ge 0$, then we say that $T$ is an \emph{$a$-contraction},
and if $(1-t)^a(T^*,T) = 0$, we say that $T$ is an \emph{$a$-isometry}.

The next result is an obvious consequence of
Theorem~\ref{thm quitar alpha tilde}.

\begin{cor}\label{cor particular eta}
Let $\op$ be an operator in $L(H)$ with spectrum $\si(\op) \ss \ol{\D}$.
Suppose that $\func$ has the form
\[
\func(t) = (1-t)^a \, \wt{\ga}(t),
\]
for some $a>0$,
where $(1-t)^a \in \CA_{\op,\R}$, and $\wt{\ga} \in \CA_{\op,\R}^0$ is positive on the
interval $[0,1]$. If $\ga(\op^*,\op) \ge 0$, 
then $\op$ is similar to an $a$-contraction.
\end{cor}

The case $a=1$ in the above corollary corresponds to \cite[Theorem~I]{BY18}. 
However, in that theorem we assumed the convergence of
the series $\sum_n|\wt{\func}_n|\|T^n\|^2$, which implies 
the convergence of the series~\eqref{ser-alpha} 
in the uniform operator topology.

The consideration of the strong operator topology here is not a mere generalization,
but it turns out to be the appropriate topology for this setting.
Indeed, using this topology, 
the results of \cite{ABY19_1} about the existence of a model
for operators $T$ satisfying that $\al(T^*,T) \ge 0$ 
are characterizations (i.e., ``if and only if'' statements).
See for instance Theorem~\ref{thm model a in 01} below.

We say that a function $\al(t) = \sum_{n=0}^{\infty} \al_n t^n$ is 
of \emph{Nevanlinna-Pick type} 
if $\al_0 = 1$ and $\al_n \le 0$ for $n \ge 1$. 
Alternatively, in this case $(1/\al)(t)$ is called a \emph{Nevanlinna-Pick kernel}. 
Whenever $k(t) = \sum_{n=0}^{\infty} k_n t^n$ has 
positive Taylor coefficients $k_n$, we denote by $\CH_k$
the weighted Hilbert space of power series
$f(t)=\sum_{n=0}^\infty f_n t^n$ with finite norm
\begin{equation}\label{eq defi space Ha}
\|f\|_{\CH_k} := \bigg(\sum_{n=0}^\infty  |f_n|^2 k_n \bigg)^{1/2}.
\end{equation}
Let $B_k$ be the
\emph{backward shift} on $\CH_k$, defined by
\begin{equation}
\label{def-B-k}
B_k f(t)= \dfrac{f(t) - f(0)}{t}.
\end{equation}
If $a>0$ and $k(t) = (1-t)^{-a}$, we denote the space $\CH_k$ 
by $\CH_a$, and the backward shift $B_k$ by $B_a$
in order to emphasize the dependence on $a$. 

It is known since Agler~\cite{Agl82} that under certain additional conditions, 
the inequalities $\al(t)\ne 0$ for $|t|<1$ and $\al(T^*, T) \ge 0$ imply 
the existence of a certain \textit{unitarily equivalent model} of $T$. 
As a consequence, in this situation, $\op$ 
(as in Theorem~\ref{thm quitar alpha tilde}) 
has a model up to similarity. 

In particular, in \cite[Theorem~1.3]{CH18}, Clou\^{a}tre and Hartz showed
the following result.
Let $\al$ be a function of Nevanlinna-Pick type. 
Suppose that
$k(t) := 1/\al(t)$ has radius of convergence $1$, its Taylor
coefficients $k_n$ are positive and satisfy $k_n / k_{n+1}\to 1$ as $n\to \infty$. 
Then $\al(T^*, T) \ge 0$ if and only if
$T$ is unitarily equivalent to a part of an operator of
the form $(B_k \otimes I_\CE) \oplus S$,
where $I_\CE$ is the identity operator on a Hilbert space $\CE$  and
$S$ is an isometry on another auxiliary Hilbert space.
By \emph{a part of an operator} we mean its restriction to an invariant subspace.
In fact, the result by Clou\^{a}tre and Hartz applies to tuples of commuting operators.  

In \cite[Theorem 1.5]{ABY19_1}, 
we complement this result by obtaining other family of functions $\al$, 
which includes non-Nevanlinna-Pick cases, so that 
all operators $T \in L(H)$ satisfying $\al(T^*, T) \ge 0$
are modeled by parts of operators of the form $B_k \otimes I_\CE$.
Moreover, we give explicit models (that is, give an explicit space $\CE$ and an explicit
isometry $S$) based on the \emph{defect operator} $D$ and the \emph{defect space} $\FD$,
given by
\begin{equation}\label{defi D and FD}
D : H \to H, \quad D := (\al(T^*,T))^{1/2},
\qquad
\FD := \ol{DH}.
\end{equation}

In \cite{ABY19_1}, we also discuss the uniqueness of this model. 
This depends on
whether $\al(1) = 0$ or $\al(1) > 0$.

The function $\al(t)=(1-t)^a$, with $a>0$, 
is of Nevanlinna-Pick type if and only if $0<a<1$.
In this case, we obtain the following fact.

\begin{thm}\label{thm model a in 01}
If $0<a<1$, then the following statements are equivalent.
\begin{enumerate}[\rm (i)]
\item
$T$ is an $a$-contraction.
\item
There exists a separable Hilbert space $\CE$ such that
$T$ is unitarily equivalent to a part of an operator $(B_a \otimes I_\CE)\oplus S$,
where $S$ is a Hilbert space isometry.
\end{enumerate}
Moreover, if \tn{(ii)} holds, then one can take for $\CE$ the space $\FD$.
\end{thm}

This result can be obtained using \cite[Theorem~1.3]{CH18}
and the explicit model obtained in \cite{ABY19_1}.
The result of Clou\^{a}tre and Hartz relies on the study of 
reproducing kernel Hilbert spaces through the
representation theory of their algebras of multipliers.
Instead of this method involving $C^*$-algebras,
here we give a direct proof of Theorem~\ref{thm model a in 01}
using approximation in Besov spaces.

\begin{cor}\label{model-up-to-sim}
Suppose that $\op$ satisfies the hypotheses of Corollary~\ref{cor particular eta}, 
where $0<a<1$. 
Then $\op$ is similar to a part of an operator of the form $(B_a \otimes I_\FD)\oplus S$,
where $S$ is a Hilbert space isometry.
\end{cor}

There are many papers on $m$-isometries for positive integers $m$.
We can mention the works
\cite{BBMP19,BMN10,BSZ18} by
Berm\'{u}dez and coauthors, \cite{Gu14} by Gu, and \cite{Ryd19} by Rydhe.
In \cite{MajMbSuciu16}, more facts about $2$-isometries are established.
The recent work \cite{Gu18Korean} discusses $m$-isometric tuples of operators
on a Hilbert space.
The $m$-contractions appear as particular cases of the families of operators
considered by Gu in \cite{Gu15}.

The study of $a$-contractions and $a$-isometries for non-integer 
$a>0$ seems to be new.
In \cite{ABY19_1} we discussed some ergodic properties of
$a$-contractions when $0<a<1$. 

The topic of $a$-contractions and $a$-isometries is closely related with 
the topic of finite differences.
Given a sequence of real numbers $\La = \{ \La_n \}_{n \ge 0}$, we denote by $\nabla \La$ the
sequence whose $n$-th term is given by $(\nabla \La)_n = \La_{n+1}$,
for $n \ge 0$.
In general, if $\be(t) = \sum \be_n t^n$ is an analytic function, we denote by
$\be(\nabla) \La$ the sequence whose $n$-th term is given by
\[
\be(\nabla) \La_n = \sum_{j=0}^{\infty} \be_j \La_{j+n},
\]
whenever the series on the right hand side converges for every $n \ge 0$.
In particular, for the functions $(1-t)^a$, with $a \in \R$, we put
\[
(1-\nabla)^a \La_n = \sum_{j=0}^{\infty} k^{-a}(j) \La_{j+n}.
\]
Here $\{ k^{-a}(j) \}$ are the so-called \emph{Cesàro numbers}; i.e., 
$k^{-a}(j)$ is the $j$-th Taylor coefficient (at the origin) of the function $(1-t)^a$. 
The above formula is the forward finite difference of order $a$ of the sequence $\La$.
For instance, for $a=1$ we get the first order finite difference 
$(1-\nabla) \La_n = \La_n - \La_{n+1}$.
We address the following two questions.

\smallskip 
\textbf{Question A.} 
Determine for which $a,b >0$ the inequality
$(1-\nabla)^a \La_n \ge 0$ (for every $n \ge 0$)
implies $(1-\nabla)^b \La_n \ge 0$ (for every $n \ge 0$).

\textbf{Question B.} 
Given $a >0$, determine the space of solutions $\La$
of the equation $(1-\nabla)^a \La = 0$.
We answer to Question A in Theorem~\ref{thm ans to Q1},
and to Question B in Theorem~\ref{thm ans to Q2}.
These two theorems rely strongly on results by Kuttner in \cite{Kut57}.
As an immediate 
consequence, we obtain the following two results 
for $a$-contractions and $a$-isometries.
The key idea is to fix a vector $x \in H$ and put $\La_n := \norm{T^n x}^2$
for $n \ge 0$.

\begin{thm}\label{thm consequence Q1}
Let $0<b<a$, where $b$ is not an integer. If $T$ is an $a$-contraction and
$T \in \weak{(1-t)^b}$,
then $T$ is a $b$-contraction.
\end{thm}

\begin{thm}\label{thm consequence Q2}
Let $a>0$, and let the integer $m$ be defined by $m<a\le m+1$.
Then the following statements are equivalent.
\begin{enumerate}[\rm (i)]
\item
$T$ is an $a$-isometry.
\item
$T$ is an $(m+1)$-isometry.
\item
For each vector $h \in H$,
there exists a polynomial $p$ of degree at most $m$
such that $\norm{T^n h}^2 = p(n)$ for every $n \ge 0$.
\end{enumerate}
\end{thm}

The contents of the paper is the following.
In Section~\ref{Section Similarity results} we prove
Theorem~\ref{thm quitar alpha tilde} using
some Banach algebras techniques.
In Section~\ref{Alternative proof in the Nevannlina-Pick case}
we give a direct proof of Theorem~\ref{thm model a in 01}
based on approximation in Besov spaces.
In Section~\ref{Section Inclusions} we
obtain answers to questions (A) and (B) above about finite differences,
and prove Theorems~\ref{thm consequence Q1}~and~\ref{thm consequence Q2}.
Finally, in Section~\ref{Section a bigger than 1} we show that a natural
conjecture about the general form of a unitarily equivalent model for $a$-contractions 
with $a>1$ is false.
The form of such model and its construction remain open. 


\section{Similarity results}
\label{Section Similarity results}

Recall that for a fixed operator $T \in L(H)$ with $\si(T) \ss \ol{\D}$, we put
\begin{equation}\label{eq defi AT}
\CA_T := \left\{ \al \in \AWC \, : \, \sum_{n=0}^{\infty} \abs{\al_n} T^{*n} T^n 
\tn{ converges in SOT} \right\}.
\end{equation}

If $X$ and $Y$ are two quantities (typically non-negative),
then $X \lesssim Y$ (or $Y \gtrsim X$) will mean that $X \le CY$ 
for some absolute constant $C > 0$.
If the constant $C$ depends on some parameter $p$,
then we write $X \lesssim_{\, p} Y$. We put $X \asymp Y$ when both
$X \lesssim Y$ and $Y \lesssim X$.

The main goal of this section is to prove Theorem~\ref{thm quitar alpha tilde}. 
We will use the following known fact.

\begin{lemma}[{\cite[Problem 120]{Hal82}}]
\label{Halmos-lemma}
If an increasing sequence $\{A_n\}$ of selfadjoint Hilbert space operators satisfies $A_n\le CI$
for all $n$, where $C$ is a constant, then $\{ A_n \}$ converges in the strong operator topology.
\end{lemma}

Using this lemma, in \cite{ABY19_1} we obtained the following fact.
\begin{prop}
[{\cite[Proposition 2.3]{ABY19_1}}]
\label{prop equivalences for Adm alpha weak}
Let $\al \in \AWC$.
Then the following statements are equivalent.
\begin{enumerate}[\rm (i)]
\item
$\al \in \CA_T$.
\item
$\sum_{n=0}^{\infty} |\al_n| \norm{T^n x}^2 <\infty$ for every $x \in H$.
\item
$\sum_{n=0}^{\infty} \abs{\al_n} \norm{T^n x}^2 \lesssim \norm{x}^2$ for every $x \in H$.
\end{enumerate}
\end{prop}

For any $\be \in \CA_T$, define
\begin{equation}\label{eq defi norm in AT}
\norm{\be}_{\CA_T} 
:= \sup_N \norm{\sum_{n=0}^{N} \abs{\be_n} T^{*n}T^n}_{L(H)}  + \norm{\be}_{\AWC}.
\end{equation}

By Proposition~\ref{prop equivalences for Adm alpha weak},
if $\be \in \CA_T$, then there exists a constant $C$ such that
\begin{equation}\label{eq equiv def AT}
\sum_{n=0}^{N} \abs{\be_n} \norm{T^n x}^2 \le C \norm{x}^2
\end{equation}
for every integer $N$ and every $x \in H$.
Indeed, it is immediate to see that one can take $C = \norm{\be}_{\CA_T}$ above,
and hence we obtain the following result.

\begin{prop}\label{prop upper bound in terms of norm beta in AT}
If $\be$ is a function in $\CA_T$, then for every vector $x \in H$ we have
\[
\sum_{n=0}^{\infty} \abs{\be_n} \norm{T^n x}^2 \le \norm{\be}_{\CA_T} \norm{x}^2.
\]
\end{prop}

\begin{nota}\label{notation succ}
If $f(t)=\sum_{n=0}^\infty f_n t^n$ and $g(t)=\sum_{n=0}^\infty g_n t^n$, we
use the notation $f \succcurlyeq g$ when $f_n \geq g_n$ for every $n \geq 0$, and
the notation  $f \succ g$ when $f \succcurlyeq g$ and $f_0 > g_0$. 
For any non-negative integer $N$, we denote by $[f]_N$ the truncation 
$\sum_{n=0}^N f_n t^n$. 
 \end{nota}

\begin{thm}\label{Thm A_T Banach algebra}
For every operator $T$ in $L(H)$ with spectrum contained in $\ol{\D}$,
$\CA_T$ is a Banach algebra with norm given by \eqref{eq defi norm in AT}.
\end{thm}

\begin{proof}
Let us prove that $\CA_T$ has the multiplicative property of algebras and its completeness
(the rest of properties for being a Banach algebra are immediate).

Let $\be$ and $\ga$ belong to $\CA_T$. We want to prove that
their product $\delta$ also belongs to $\CA_T$.
Note that $\delta \in \AWC$.
By Proposition~\ref{prop equivalences for Adm alpha weak},
we just need to prove the existence of a constant $C>0$ such that
\begin{equation}\label{eq product be ga algebra}
\sum_{n=0}^{N} \abs{\delta_n} \norm{T^n x}^2 \le C \norm{x}^2
\end{equation}
for every non-negative integer $N$ and every vector $x \in H$. 

So take any $x \in H$, and let $N \ge 0$. 
Put
\[
\abs{\be} (t) := \sum_{n=0}^{\infty} \abs{\be_n} t^n, \quad \abs{\ga} (t) 
:= \sum_{n=0}^{\infty} \abs{\ga_n} t^n, \quad \wt{\delta}(t) = \abs{\be}(t) \cdot \abs{\ga}(t).
\]
Hence
\[
\wt{\de}_n = \sum_{j=0}^{n} \abs{\be_j} \abs{\ga_{n-j}} \ge \abs{\de_n}.
\]
Therefore \eqref{eq product be ga algebra} will follow if we prove the existence of 
a positive constant $C$ such that
\begin{equation}\label{eq product be ga algebra 2}
\sum_{n=0}^{N} \wt{\delta}_n \norm{T^n x}^2 \le C \norm{x}^2.
\end{equation}
Using that $[\abs{\be}]_N \cdot \abs{\ga} \succ [\wt{\de}]_N$
(recall Notation~\ref{notation succ}),
we have that
\[
\begin{split}
\sum_{n=0}^{N} \wt{\delta}_n \norm{T^n x}^2
&= \ip { [\wt{\de}]_N (T^*,T) x } { x } \le \ip {([\abs{\be}]_N \cdot \abs{\ga})(T^*,T) x} {x} \\
&= \sum_{n=0}^{N} \ip {\abs{\ga}(T^*,T) \abs{\be_n} T^n x} {T^n x}
\le \norm{\abs{\ga}(T^*,T)} \sum_{n=0}^{N} \abs{\be_n} \norm{T^n x}^2 \\
&\le \norm{\abs{\ga}(T^*,T)} \norm{\be}_{\CA_T} \norm{x}^2.
\end{split}
\]
Note that the operator $\abs{\ga} (T^*, T)$ belongs to $L(H)$ because $\ga \in \CA_T$. 
Now we can take $C = \norm{\abs{\ga} (T^*, T)} \norm{\be}_{\CA_T}$ 
(which depends neither on $N$ nor on $x$), 
and \eqref{eq product be ga algebra 2} follows.

Let us prove now the completeness of $\CA_T$. 
Let $\{ \be^{(k)} \}_{k\ge 0}$ be a Cauchy sequence in $\CA_T$. 
In other words,
\begin{equation}\label{eq product be ga algebra 3}
\norm{\be^{(k)} - \be^{(\ell)}}_{\CA_T} \to 0 \quad \tn{ when } k, \ell \to \infty.
\end{equation}
We want to prove the existence of a function $\be$ in $\CA_T$ 
such that the sequence $\{ \be^{(k)} \}_{k\ge 0}$
converges to $\be$ in the norm $\norm{\cdot}_{\CA_T}$.
Since $\norm{\cdot}_{\CA_T} \ge \norm{\cdot}_{\AWC}$ and
$\AWC$ is complete,
there exists a function $\be$ in $\AWC$ such that $\be^{(k)}$ converge
to $\be$ in the norm of $\AWC$. Now fix $\ep > 0$. 
Then by~\eqref{eq product be ga algebra 3}, 
there exists an integer $M$ such that
\[
\norm{\be^{(k)} - \be^{(\ell)}}_{\CA_T} < \ep
\]
if $k, \ell \ge M$. In other words, for every $N$ we have
\[
\norm{\sum_{n=0}^{N} \abs{\be_n^{(k)} - \be_n^{(\ell)}} T^{*n} T^n}_{L(H)} 
+ \norm{\be^{(k)} - \be^{(\ell)}}_{\AWC} < \ep.
\]
Now taking the limit when $\ell \to \infty$ above, we obtain
\[
\norm{\sum_{n=0}^{N} \abs{\be_n^{(k)} - \be_n} T^{*n} T^n}_{L(H)} 
+ \norm{\be^{(k)} - \be}_{\AWC} \le \ep,
\]
so $\norm{\be^{(k)} - \be}_{\CA_T} \le \ep$ (if $k \ge M$).
\end{proof}

Recall that in the Introduction we defined $\CA_T^{0}$ as
\begin{equation}
\label{charn_A_T^0}
\CA_T^{0} =
\bigg\{
\be \in \CA_T \, : \, \sum_{n=0}^{\infty} \abs{\be_n} T^{*n} T^n
\tn{ converges in norm of } L(H)
\bigg\}.
\end{equation}

\begin{prop}\label{prop AT0 conv in norm}
$\CA_T^{0}$ is the closure of the polynomials in $\CA_T$.
In particular, it is a separable closed subalgebra of $\CA_T$.
\end{prop}

\begin{proof}
Let us denote provisionally by $\mathcal{CP}$ the closure of the polynomials in $\CA_T$.
Let $\be \in \mathcal{CP}$ and fix $\ep > 0$. There exists a polynomial $p$ such that
\[
\norm{\be - p}_{\CA_T} < \ep/2.
\]
Let $N$ be an integer larger that the degree of $p$. Since $p=[p]_N$,
\[
\begin{split}
\bigg\| \sum_{n=N+1}^{\infty} \abs{\be_n} T^{*n} T^n \bigg\|_{L(H)}
&\le \norm{\be - [\be]_N}_{\CA_T} \le \norm{\be-p}_{\CA_T} + \norm{[\be-p]_N}_{\CA_T} \\
&\le 2 \norm{\be-p}_{\CA_T} < \ep.
\end{split}
\]
Hence $\sum_{n=0}^{\infty} \abs{\be_n} T^{*n} T^n$ converges uniformly in $L(H)$.
This proves the inclusion $\mathcal{CP} \ss \CA_T^{0}$.

The inclusion $\CA_T^{0} \ss \mathcal{CP}$ is immediate.
Indeed, any $\be \in \CA_T^{0}$ can be approximated
in $\CA_T$ by the truncations $[\be]_N$.
\end{proof}

\begin{prop}\label{thm gamma n less than beta n then gamma also in AT}
Let $\be$ be a function in $\CA_T$.
\begin{enumerate}[\rm (i)]
\item
If $\abs{\ga_n} \le \abs{\be_n}$ for every $n$, 
then $\ga$ also belongs to $\CA_T$ and moreover
$\norm{\ga}_{\CA_T} \le \norm{\be}_{\CA_T}$.
\item
If $\ga_n = \be_n \tau_n$, where $\tau_n \to 0$, 
then $\ga$ also belongs to $\CA_T^{0}$.
\end{enumerate}
\end{prop}

\begin{proof}
(i) is immediate. For the proof of (ii), put
\[
C_N := \max_{n \ge N} \abs{\tau_n},
\]
for each positive integer $N$. Then
\[
\norm{ \sum_{n=N}^{\infty} \abs{\ga_n} T^{*n} T^n } 
\le C_N \norm{ \sum_{n=N}^{\infty} \abs{\be_n} T^{*n} T^n } 
\le C_N \norm{\be}_{\CA_T} \to 0,
\]
and therefore $\ga$ belongs to $\CA_T^0$.
\end{proof}

\begin{prop}
The characters of $\CA_T^{0}$ are precisely the evaluation functionals at points of $\ol{\D}$.
\end{prop}

\begin{proof}
Let $\chi$ be a character of $\CA_T^{0}$ (i.e., it is a multiplicative
bounded linear functional on $\CA_T^{0}$ that satisfies $\chi(1) = 1$). For
the function $t$ in $\CA_T^{0}$, let us put
\[
\la := \chi(t) \in \C.
\]
Therefore $\chi$ sends every polynomial $p(t)$ into the number $p(\la)$. Let us prove
now that $\abs{\la} \le 1$. Using the obvious fact
\[
\abs{\la} = (\abs{\la}^n)^{1/n} = \abs{\chi(t^n)}^{1/n}
\]
we obtain
\[
\abs{\la} = \limsup_{n \to \infty} \abs{\chi(t^n)}^{1/n} 
\le \limsup_{n \to \infty} \norm{t^n}_{\CA_T^{0}}^{1/n} \le 1,
\]
where we have used that $\norm{\chi} = 1$ (since it is a character) and that
the spectral radius of $T$ is less or equal than $1$.

By the continuity of $\chi$ and the density of the polynomials in $\CA_T^{0}$
we obtain that $\chi$ sends every function
$f(t)$ in $\CA_T^{0}$ to $f(\la)$.
\end{proof}

\begin{cor}\label{cor inverse function in AT0}
If $\be \in \CA_T^{0}$ with $\be(t) \neq 0$ for every $t \in \ol{\D}$, then $1/\be \in \CA_T^{0}$.
\end{cor}

Indeed, the condition on $h$ means that $\chi(h) \neq 0$ for every character $\chi$.
Hence the result follows using the Gelfand Theory.

\begin{thm}\label{thm analog key lemma}
Let $T \in L(H)$ with $\si(T) \ss \ol{\D}$ and let $f \in \CA_{T,\R}^{0}$.
If $f(t) > 0$ for every $t \in [0,1]$,
then there exists a function $g \in \CA_{T,\R}^{0}$ such that
$g \succ 0$ and $fg \succ 0$.
\end{thm}

As an immediate consequence of this theorem, we obtain the following result.

\begin{cor}
\label{cor to thm analog key lemma}
If $f \in \AW$ satisfies $f(t) > 0$ for every $t \in [0,1]$,
then there exists a function $g \in \AW$ such that $g \succ 0$ and $fg \succ 0$.
\end{cor}

\begin{proof}
Take as $T$ any power bounded operator. Then $\si(T) \ss \ol{\D}$ and
$\CA_T^{0}$ is precisely the Wiener algebra $\AWC$.
Then apply Theorem~\ref{thm analog key lemma}.
\end{proof}

We denote by
$\mathcal{H}(\ol{\D})$ the set of
all analytic functions in a neighborhood of
$\ol{\D}$.
For the proof of Theorem~\ref{thm analog key lemma} we need
the following result.

\begin{lemma}
[{\cite[Lemma~2.1]{BY18}}]
\label{Lema_Berns_Polya}
If $q$ is a real polynomial such that $q(t)>0$ for every $t \in
[0,1]$, then there exists a rational function $u \in
\mathcal{H}(\ol{\D})$ such that $u\succ 0$ and $uq \succ 0$.
\end{lemma}

\begin{proof}[Proof of Theorem~\ref{thm analog key lemma}]
By hypothesis, there exists a positive number $\ep$ such that 
$f(t) > \ep > 0$ for every $t \in [0,1]$.

\begin{claims}
There exists a positive integer $N$ such that
\[
\norm{\sum_{n=N+1}^{\infty} f_n t^n}_{\CA_T^0} < \ep / 2.
\]
\end{claims}

Indeed, since $f \in \AW$ we deduce that there exists a positive integer $N_1$ 
such that for every $M \ge N_1$ we have
\[
\sum_{n=M+1}^{\infty} \abs{f_n} < \ep / 4.
\]
Using Proposition \ref{prop AT0 conv in norm} 
we obtain the existence of a positive integer $N_2$ such that 
for every $M \ge N_2$ we have
\[
\norm{\sum_{n=M+1}^{\infty} \abs{f_n} T^{*n} T^n}_{L(H)} < \dfrac{\ep}{4}.
\]
Hence, taking $N := \max\{N_1, N_2\}$ the claim follows.

In particular, we get
\[
\sum_{n=N+1}^{\infty} \abs{f_n} < \ep/2.
\]
Put
\[
f_N(t) := \sum_{n=0}^{N} f_n t^n - \dfrac{\ep}{2}, 
\quad \quad 
h(t) := \dfrac{\ep}{2} + \sum_{n \ge N+1; f_n < 0} f_n t^n.
\]
Note that $f_N$ is a polynomial. It is easy to see that $h \in \CA_T^0$. Since
\[
f_N(t) > \ep - \dfrac{\ep}{2} - \dfrac{\ep}{2} = 0
\]
for every $t \in [0,1]$, we can use
Lemma~\ref{Lema_Berns_Polya} to obtain a function $u \in \mathcal{H}(\ol{\D})$ such that
$u \succ 0$ and $u f_N \succ 0$. 
Note that it is immediate that all the functions in $\mathcal{H}(\ol{\D})$ also belong
to the algebra $\CA_T^0$.

For every $t \in \ol{\D}$ we have
\[
\bigg| \sum_{n \ge N+1; \\ f_n < 0} f_n t^n \bigg| 
\le \sum_{n=N+1}^{\infty} \abs{f_n} < \dfrac{\ep}{2}.
\]
Hence $h$ does not vanish on $\ol{\D}$, 
and therefore using Corollary \ref{cor inverse function in AT0} 
we obtain
\[
v := h^{-1} \in \CA_T^0.
\]
Note that $v \succ 0$.
Finally, put $g := uv$. Obviously $g \succ 0$, 
and since $f \succcurlyeq f_N + h$ we obtain that
\[
gf \succcurlyeq g(f_N+h) = vuf_N + u \succ 0,
\]
as we wanted to prove.
\end{proof}

\begin{prop}\label{prop fnTBT SOT}
Let $f \in \CA_{T,\R}$ and let $B \in L(H)$ be a non-negative operator.
Then the operator series
\begin{equation}
\sum_{n=0}^{\infty} f_n T^{*n} B T^n
\end{equation}
converges in the strong operator topology in $L(H)$.
\end{prop}

\begin{proof}
First we observe that the above series
converges in the weak operator topology.
Indeed, for every pair of vectors $x,y \in H$ we have
\[
\begin{split}
\Big| \Big\langle\big( \sum_{n=N}^{M} f_nT^{*n} B T^n \big)x, y\Big\rangle\Big|
&= 
\Abs{ \sum_{n=N}^{M} f_n \ip{BT^nx}{T^ny} } 
\le 
\sum_{n=N}^{M} \abs{f_n} \norm{B} \norm{T^n x} \norm{T^n y} \\
&\le 
\sum_{n=N}^{M} \abs{f_n} \norm{B} (\norm{T^n x}^2 
+ \norm{T^n y}^2 ) \to 0 \quad (N,M \to \infty),
\end{split}
\]
and the statement follows.
Next, put
\[
f_n^+ := \max\{f_n, 0\}, \quad f_n^- := \max\{-f_n, 0\}.
\]
By the above, the series
\[
\sum_{n=0}^{\infty} f_n^+ T^{*n} B T^n 
\quad \tn{ and } \quad 
\sum_{n=0}^{\infty} f_n^- T^{*n} B T^n
\]
converge in the weak operator topology in $L(H)$. By Lemma~\ref{Halmos-lemma}
these series also converges in SOT. 
Since $f_n = f_n^+ - f_n^-$, we also obtain the convergence
in SOT of the series $\sum f_n T^{*n} B T^n$.
\end{proof}

\begin{defi}
As a consequence of Proposition \ref{prop fnTBT SOT}, for every $f \in \CA_{T,\R}$
and every non-negative operator $B \in L(H)$ we can define
\[
f(T^*,T)(B) := \sum_{n=0}^{\infty} f_n T^{*n} B T^n,
\]
where the convergence is in SOT.
In particular, when $B$ is the identity operator in $L(H)$,
\[
f(T^*,T) = f(T^*,T)(I) = \sum_{n=0}^{\infty} f_n T^{*n} T^n.
\]
\end{defi}

\begin{rem}
Observe that, by applying 
Proposition~\ref{prop upper bound in terms of norm beta in AT} 
to the vector $x = B^{1/2}h$, we get
\begin{equation}\label{eq upper bound for norm of fT*TB}
\norm{f(T^*,T)(B)} \lesssim \norm{B} \norm{f}_{\CA_T}.
\end{equation}
\end{rem}

\begin{lemma}\label{Lemma 24 CAOT}
Let $B \in L(H)$ be a non-negative operator and let $f, g \in \CA_{T,\R}$.
Put $h:= fg$. Then
\begin{enumerate}[\rm (i)]
\item
$h(T^*,T)(B) = g(T^*,T)(f(T^*,T)(B))$;
\item
$h(T^*,T) = g(T^*,T)(f(T^*,T))$.
\end{enumerate}
\end{lemma}

\begin{proof}
Note that (ii) is just an application of (i) for $B=I$. Let us start proving (i) for the case
where all the coefficients $f_n$ and $g_n$ are non-negative. 
In this case, both parts of (i) are well defined by Proposition \ref{prop fnTBT SOT}. 
Then
\[
\begin{split}
g(T^*,T)(f(T^*,T)(B))
&= \sum_{n=0}^{\infty} g_n T^{*n} \Big( \sum_{m=0}^{\infty} f_m T^{*m} B T^m \Big) T^n \\
&= \sum_{n=0}^{\infty} \sum_{m=0}^{\infty} g_n f_m T^{*n+m} B T^{n+m} \\
& \stackrel{(\star)}{=} \sum_{k=0}^{\infty} \bigg( \sum_{n+m=k} f_m g_n \bigg) T^{*k} B T^k
= h(T^*,T)(B),
\end{split}
\]
where in ($\star$)
all the series are understood in the sense of the SOT convergence.
To justify it, it suffices to pass to quadratic forms and
to use that $f_m$ and $g_n$ are all non-negative. 
Finally, the general case $f,g \in \CA_T$ can be
derived from the previous one by linearity and using the decompositions
\[
f = f^+ - f^-, \quad g = g^+ - g^-,
\]
where $f^+, f^-, g^+$ and $g^-$ have non-negative Taylor coefficients.
\end{proof}

\begin{proof}[Proof of Theorem~\ref{thm quitar alpha tilde}]
By Theorem \ref{thm analog key lemma},
there exists a function $g \in \CA_{T,\R}^{0}$ such that $g \succ 0$ and
$h := \wt{\ga} g \succ 0$.
Then $\ga(t) g(t) = h(t) \al(t)$ and, by Lemma~\ref{Lemma 24 CAOT}~(ii), we get
\begin{equation}
\label{des-operadores}
\sum_{n=0}^{\infty} \al_n \op^{*n} h(\op^*,\op) \op^n =
\sum_{n=0}^{\infty} g_n \op^{*n} \ga(\op^*,\op) \op^n \geq 0.
\end{equation}
Define an operator $B>0$ by $B^2 := h(\op^*,\op) \ge \ep I > 0$ (for some $\ep > 0$).
Then we have
\[
\sum_{n=0}^{\infty} \al_n \norm{B\op^n x}^2 \ge 0
\]
for every $x \in H$. Denote $y := Bx$ and put
\[
T := B\wh{T}B^{-1}.
\]
We get that
\[
\sum_{n=0}^{\infty} \al_n \norm{T^n y}^2 \ge 0
\]
for every $y \in H$. Therefore $T$ is similar to $\wh{T}$ and $\al(T^*, T) \ge 0$.
\end{proof}

In particular this proves Corollary~\ref{cor particular eta}.
The case $a=1$ in this corollary can be compared with
\cite[Theorem~3.10]{Mul88}. Notice, however, that in that theorem by M\"{u}ller
it is assumed that $T$ is a contraction.


\section{A direct proof of Theorem~\ref{thm model a in 01}}
\label{Alternative proof in the Nevannlina-Pick case}

Our proof will be based on approximation in Besov spaces.

\begin{proof}[Proof of Theorem~\ref{thm model a in 01}]
The implication (ii) $\Rightarrow$ (i) is straightforward,
so we focus on the implication (i) $\Rightarrow$ (ii).
Let $T$ be an $a$-contraction. We divide the proof into three steps.
First, we define the operator $S$,
then we prove that $S$ is an isometry,
and finally we prove that $T$ is unitarily equivalent of a part of
$(B_a \otimes I_\mathfrak{D})\oplus S$.

We put $\al(t) := (1-t)^a$ and $k(t) := (1-t)^{-a} \succ 0$.
Recall that $D$ is the non-negative square root of
the operator $(1-t)^a(T^*,T) \ge 0$.

\begin{step1}
Note that
\begin{equation}\label{eq norm Dx}
\norm{Dx}^2 = \sum_{j=0}^{\infty} \al_j \norm{T^jx}^2 \quad \quad (\forall x \in H).
\end{equation}
Changing $x$ by $T^nx$ in \eqref{eq norm Dx} we obtain a more general formula:
\begin{equation}\label{eq norm DTj x}
\norm{DT^nx}^2 = \sum_{j=0}^{\infty} \al_j \norm{T^{j+n}x}^2 
\quad \quad (\forall x \in H, \, \forall n \ge 0).
\end{equation}
Multiplying \eqref{eq norm DTj x} by $k_n$ and summing for $n=0,1,\ldots, N$ 
(for some fixed $N \in \N$), we get the following equation

\begin{equation}\label{eq. B3 1}
\norm{x}^2 =
\sum_{n=0}^{N} k_n \norm{DT^n x}^2 + \sum_{m=N+1}^{\infty} \norm{T^m x}^2 \rho_{N,m}
\quad (\forall x \in H),
\end{equation}
where
\begin{equation}\label{eq defi rho N ell}
\rho_{N,m} = \sum_{n=N+1}^{m} k_n \al_{m-n} = -\sum_{j=0}^{N} k_j \al_{m-j}
\quad \quad (1 \le N+1 \le m).
\end{equation}

Since $0<a<1$, we have that $\al_n < 0$ for every $n \ge 1$ (and $\al_0 = 1$). 
Note that  in the last sum in \eqref{eq defi rho N ell} all the $k_j$'s
are positive and all the $\al_n$'s are negative, because 
$\al_0$ does not appear there. We obtain
that $\rho_{N,m} > 0$.  
Therefore, by \eqref{eq. B3 1} we have
\[
\norm{x}^2 \ge \sum_{n=0}^{N} k_n \norm{DT^n x}^2.
\]
Hence the series with positive terms $\sum k_n \norm{DT^n x}^2$ converges, and taking
limits in \eqref{eq. B3 1} when $N \to \infty$ we obtain
\begin{equation}\label{eq existence of limit with rho N ell}
\exists \lim_{N \to \infty} \sum_{m=N+1}^{\infty} \norm{T^m x}^2 \rho_{N,m} 
= \norm{x}^2 - \sum_{n=0}^{\infty} k_n \norm{DT^n x}^2 \ge 0.
\end{equation}
We are going to define a new semi-inner product on our
Hilbert space $H$ via
\begin{equation}\label{eq. B3 2}
[x,y] := \lim_{N \to \infty} \sum_{m=N+1}^{\infty} \ip{T^m x}{T^m y} \rho_{N,m}.
\end{equation}
By \eqref{eq existence of limit with rho N ell}, $[x,x]$ is correctly defined,
and $[x,x] = \left\langle Ax,x \right\rangle$, where $A$ is a self-adjoint operator with
$0 \le A \le I$. Hence, by the polarization formula $[x,y]$ is correctly defined for any
$x,y \in H$, and $[x,y] = \left\langle Ax,y \right\rangle$.

Let $E:= \left\{ x \in H \, : \, [x,x] = 0 \right\}$.
It is a closed subspace of $H$.
Put $\wh{H} := H / E$.
For any vector $x \in H$, we denote by $\wh{x}$ its 
equivalence class in $\wh{H}$.
Note that $\wh{H} := H / E$ is a new Hilbert
space with norm $\Norm{\cdot}$ given by
\begin{equation}\label{eq. B3 3}
\Norm{\wh{x}}^2 
= \lim_{N \to \infty} \sum_{m=N+1}^{\infty} \norm{T^m x}^2 \rho_{N,m} 
= \norm{x}^2 - \sum_{n=0}^{\infty} k_n \norm{DT^n x}^2.
\end{equation}
We set $S$ to be the operator on $\wh{H}$, 
given by $S \wh{x} := \wh{Tx}$ for every $x \in H$.
\end{step1}

\begin{step2}
Let us see now that the equality
$\Norm{\wh{Tx}} = \Norm{\wh{x}}$ holds for every $x \in H$.
Observe that this will imply, in particular, that $S$ is well-defined.

Indeed, note that
\begin{equation}\label{eq. B3 4}
\sum_{n=0}^{\infty} k_n \norm{DT^n x}^2 
= \lim_{r \to 1^-} \sum_{n=0}^{\infty} r^n k_n \norm{DT^n x}^2
\end{equation}
since the RHS is an increasing function of $r$.
Using \eqref{eq norm Dx}, \eqref{eq. B3 3} and \eqref{eq. B3 4} we obtain that
\begin{equation}\label{eq. B3 5}
\begin{split}
\Norm{\wh{x}}^2 &= \norm{x}^2 - \lim_{r \to 1^-} \sum_{n=0}^{\infty} r^n k_n \norm{DT^n x}^2 \\
&= 
\norm{x}^2 - \lim_{r \to 1^-} \sum_{n=0}^{\infty} 
\sum_{j=0}^{\infty} r^n k_n \al_j \norm{T^{n+j}x}^2 \\
&= 
\norm{x}^2 - \lim_{r \to 1^-} \sum_{m=0}^{\infty} \norm{T^m x}^2 u_r(m),
\end{split}
\end{equation}
where
\begin{equation}\label{eq. B3 7}
u_r(m) := \sum_{n+j = m; \, n,j\ge 0}\; r^n k_n \al_j. 
\end{equation}
For each $m$, $u_r(m)$ is a continuous function of $r \in [0,1]$. 
We also have  $u_1(m) = 0$ for $m \ge 1$ and $u_r(0) = 1$ for $r \in [0,1]$. 
Moreover, since $u_r(1) = a(1-r)$, we have $u_r(1) \to 0$ as $r \to 1^-$. 
Therefore
\begin{equation}\label{eq. B3 8}
\Norm{\wh{Tx}}^2 - \Norm{\wh{x}}^2 
= \lim_{r \to 1^-} \sum_{m=2}^{\infty} \norm{T^m x}^2 [u_r(m) - u_r(m-1)].
\end{equation}
We have to prove that it is zero for any $x \in H$.

\begin{claims}
There exists a constant $C > 0$ independent of $r$ and $m$ such that
\begin{equation}\label{eq. B3 9}
\abs{u_r(m) - u_r(m-1)} \le \frac{C}{m^{1+a}} \quad (\forall m \ge 2).
\end{equation}
\end{claims}
Indeed,
it is easy to check that
\[
\sum_{m=0}^{\infty} u_r(m) t^m = \frac{(1-t)^a}{(1-rt)^a} 
\quad \quad (\abs{t} < 1 \tn{ and } 0<r<1).
\]
Multiplying by $(1-t)$, we get
\[
1 + \sum_{m=1}^{\infty} [u_r(m) - u_r(m-1)] t^m = \frac{(1-t)^{a+1}}{(1-rt)^a} =: f_r(t).
\]
Hence \eqref{eq. B3 9} is equivalent to
\begin{equation}\label{eq. B3 10}
\abs{\wh{f}_r(n)} \leq \frac{C}{(n+1)^{1+a}} \quad (\forall n),
\end{equation}
where the constant $C$ has to be 
independent of $r$ and $n$. 
Equivalently, we want to prove that the Fourier coefficients of
\begin{equation}\label{eq. B3 11}
\sum_{n=0}^{\infty} (1+n)^{1+a} \wh{f}_r(n) z^n = I_{-1-a} f_r
\end{equation}
are uniformly bounded, where for $\be \in \R$,
\[
I_\be h (z) := \sum_{j=0}^{\infty} (1+j)^{-\be} \wh{h}(j) z^j,
\]
as in \cite[p. 737]{Pel03}.
We will prove the even stronger result
\begin{equation}\label{eq. B3 12}
\norm{I_{1-a} f_{r}^{''}}_{H^1} \le C,
\end{equation}
where the constant $C$ does not depend on $r$.
Here, $H^p$ denotes the classical Hardy space of the unit disc.
Since for any $\be>0$
\begin{equation}\label{eq. B4 1}
I_\be h \in H^1 
\quad \Longleftrightarrow \quad 
\int_{\T} \left( \int_{0}^{1} \abs{h(\rho \zeta)}^2 (1-r)^{2\be-1} \, d{\rho}  \right)^{1/2} 
\, \abs{d{\zeta}} < \infty
\end{equation}
(see \cite[p. 737]{Pel03}), we obtain that \eqref{eq. B3 12} is equivalent to
\begin{equation}\label{eq. B3 13}
\sup_{r\in[0,1]}
\int_{\T} \left[ \int_{0}^{1} \abs{f_{r}^{''}(\rho \zeta)}^2 (1-\rho)^{1-2a} \, d\rho \right]^{1/2} 
\, \abs{d{\zeta}} \le C.
\end{equation}
It is immediate to check that $f_{r}^{''}$ can be represented as
\[
f_{r}^{''}(t) = \sum_{j=0}^{2} c_j(r) (1-t)^{a-1+j} (1-rt)^{-a-j},
\]
where the $c_j$'s are bounded functions. Using that $\abs{1-rt} \ge M \abs{1-t}$
for a certain constant $M$ (for $\abs{t} < 1$ and $0<r<1$), we obtain that
\[
\abs{f_{r}^{''}(t)} \le C_1 \abs{1-t}^{-1}.
\]
Therefore we just need to prove that
\begin{equation}\label{eq. B3 14}
\int_{\T} \left[ \int_{0}^{1} \abs{1-\rho\zeta}^{-2} (1-\rho)^{1-2a} \, d\rho \right]^{1/2} 
\, \abs{d{\zeta}} < \infty
\end{equation}
because now we do not have dependence on $r$.
By \eqref{eq. B4 1}, this is equivalent to
\[
I_{1-a} g \in H^1, \quad \tn{where} \quad  g(t) = (1-t)^{-1}.
\]
Hence it remains to check that
\[
\sum_{n=0}^{\infty} (n+1)^{a-1} t^n \in H^1.
\]
For the sake of an easier notation, we will prove an equivalent statement
\begin{equation}\label{eq inverse 1 - t in H1}
\sum_{n=1}^{\infty} n^{a-1} t^n \in H^1.
\end{equation}
Recall that $k(t) = (1-t)^{-a}$. 
Then, by \cite[Vol. I, p. 77 (1.18)]{Zygmund} we have
\[
k_n = \dfrac{1}{\Gamma(a)} n^{a-1} \left( 1 + O \left( \dfrac{1}{n} \right) \right).
\]
Therefore
\[
n^{a-1} = \Gamma(a) k_n - n^{a-1} v_n, \quad \tn{ where } \abs{v_n} \lesssim n^{-1}.
\]
Since $0<a<1$, we know that $\sum k_n t^n = (1-t)^{-a}$ belongs to $H^1$. 
Since the function $\sum n^{a-1} v_n t^n$ belongs to $H^2$, 
\eqref{eq inverse 1 - t in H1} follows.
This finishes the proof of the claim.

Finally, since $T$ is an $a$-contraction, we know that
\begin{equation}\label{eq series al_n sim to}
\sum_{n=1}^{\infty} \frac{1}{n^{1+a}} \norm{T^n x}^2
\, \,  \sim \, \,
\sum_{n=0}^{\infty} |\al_n| \norm{T^n x}^2
\quad
\tn{converges for every } x \in H.
\end{equation}
By the claim, we can estimate the series in the right hand side of
\eqref{eq. B3 8} as
\[
\sum_{m=2}^{\infty} \norm{T^m x}^2 \abs{u_r(m) - u_r(m-1)}
\lesssim
\sum_{m=2}^{\infty} \dfrac{\norm{T^m x}^2}{m^{1+a}} < \infty.
\]
Hence, Lebesgue's Dominated Convergence Theorem allows us to exchange
the limit with the sign of sum in \eqref{eq. B3 8}. Using 
that $u_1(j)=0$ for every $j \ge 1$,
we obtain that $\Norm{\wh{Tx}}^2 = \Norm{\wh{x}}^2$ for any $x \in H$.
Hence $S$ is well defined and it is an isometry.
\end{step2}

\begin{step3}
As usual, $\FD$ is the closure of $DH$.
Let
\[
G : H \to (\CH_a \otimes \FD) \oplus \wh{H}, \quad Gx := (\{Dx,DTx, DT^2x, \ldots \}, \wh{x}).
\]
By \eqref{eq. B3 3}, $G$ is an isometry. It is immediate that
\[
((B_a \otimes I_\FD) \oplus S) G = G T.
\]
Hence we get that $T$ is unitarily equivalent to a part of 
$(B_a \otimes I_\FD) \oplus S$. \qedhere
\end{step3}
\end{proof}

In Section~\ref{Section a bigger than 1}, we will see that an analogue of 
Theorem~\ref{thm model a in 01} is false if $a>1$.


\section{Inclusions for classes of $a$-contractions and $a$-isometries}
\label{Section Inclusions}

In this section we prove 
Theorem~\ref{thm consequence Q1} and Theorem~\ref{thm consequence Q2}.
Let us begin with the following three results given in \cite{Kut57} by Kuttner.

\begin{thm}
[{\cite[Theorem 3]{Kut57}}]
Let $\si > -1$, $\si$ not an integer.
Let $\La = \{ \La_n \}_{n \ge 0}$ be a sequence of real numbers.
Suppose that
$(1-\nabla)^\si \La$
is well defined. If
\begin{equation}\label{eq Kuttner 1}
s \ge \si, \quad r+s > \si, \quad \tau \ge \si-r, \quad \tau \ge 0,
\end{equation}
or if
\begin{equation}\label{eq Kuttner 2}
s \ge \si, \quad r+s = \si, \quad \tau > \si-r, \quad \tau \ge 0,
\end{equation}
then
\[
(1-\nabla)^r [(1-\nabla)^s \La]_n
=
\sum_{j=0}^{\infty} k^{-r}(j) \left( \sum_{m=0}^{\infty} k^{-s}(m) \La_{m+j+n} \right)
\]
is summable $(C,\tau)$ to $(1-\nabla)^{r+s} \La_n$.
\end{thm}

Recall that a sequence $\{ s_n \}_{n \ge 0}$ is \emph{summable $(C,\tau)$} to a limit $s$ if 
the sequence of its \emph{Cesàro means of order $\tau$}
\[
C_\tau s_n := \dfrac{1}{k^{\tau+1}(n)}\sum_{j=0}^{n} k^\tau(n-j) s_j
\]
converges to $s$ as $n$ goes to infinity. 
Summability $(C,0)$ is ordinary convergence, and summability $(C,\tau)$ of a sequence 
implies summability $(C,\tau')$ to the same limit for any $\tau' > \tau$. 

\begin{thm}
[{\cite[Theorem A]{Kut57}}]
\label{Kuttner Thm A}
Let $s > -1$ and $r \ge 0$. 
If $\La = \{\La_n\}_{n \ge 0}$ is a sequence of real numbers, 
then
\begin{equation*}
(1-\nabla)^{r+s} \La_n = (1-\nabla)^r [(1-\nabla)^s \La]_n
\end{equation*}
for every $n \ge 0$, 
whenever the right hand side is well defined.
\end{thm}

\begin{thm}
[{\cite[Theorem B]{Kut57}}]
\label{Kuttner Thm B}
Let $s > -1, r+s > -1$ and $r+s$ be non-integer.
If $\La = \{\La_n\}_{n \ge 0}$ is a sequence of real numbers, then
\begin{equation*}
(1-\nabla)^{r+s} \La_n = (1-\nabla)^r [(1-\nabla)^s \La]_n
\end{equation*}
for every $n \ge 0$, whenever both sides above are well defined.
\end{thm}

Now we give an answer to Question A from the Introduction.
As a consequence, we will derive Theorem~\ref{thm consequence Q1}.

\begin{thm}\label{thm ans to Q1}
Let $\La = \{ \La_n \}_{n \ge 0}$ be a sequence of real numbers,
and let $0<b<a$, where $b$ is not an integer. 
If $(1-\nabla)^a \La_n \ge 0$ for every $n \ge 0$
and $(1-\nabla)^b \La$ is well defined, 
then $(1-\nabla)^b \La_n \ge 0$ for every $n \ge 0$.
\end{thm}

\begin{proof}
Let $0<b<a$, and let $\La = \{ \La_n \}_{n \ge 0}$ be 
a sequence of real numbers such that $(1-\nabla)^a \La_n \ge 0$ 
for every $n \ge 0$, and $(1-\nabla)^b \La$ is well defined.
Putting
\[
\si = b, \quad s = a, \quad r = b-a, \quad \tau=[a]+1
\]
in \eqref{eq Kuttner 2} 
(where $[a]$ denotes the biggest integer less than or equal to $a$), 
we obtain that the series
\begin{equation}\label{eq Kuttner thm3 Cesaro Sum}
(1-\nabla)^{b-a} [(1-\nabla)^a \La]_n
=
\sum_{j=0}^{\infty} k^{a-b}(j) \left( \sum_{m=0}^{\infty} k^{-a}(m) \La_{m+j+n} \right)
\end{equation}
is summable $(C,[a]+1)$ to $(1-\nabla)^b \La_n$.

Note that $k^{a-b}(m) \ge 0$ for every $m \ge 0$ 
(because $a-b > 0$) 
and the series in parenthesis in 
the right hand part of \eqref{eq Kuttner thm3 Cesaro Sum} are non-negative for every $j\ge 0$. 
We deduce that
$(1-\nabla)^b \La_n \ge 0$ for every $n \ge 0$, as we wanted to prove.
\end{proof}

\begin{proof}[Proof of Theorem~\ref{thm consequence Q1}]
Let $0<b<a$, where $b$ is not an integer, and let $T$ be an $a$-contraction
such that $T \in \weak{b}$.
Fix $x \in H$ and put $\La_n := \norm{T^n x}^2$, for $n \ge 0$.
If we show that
\begin{equation}
\label{eq La_n an contr}
(1-\nabla)^b \La_n =
\sum_{j=0}^{\infty} k^{-b}(j) \La_{j+n} =
\sum_{j=0}^{\infty} k^{-b}(j) \norm{T^{j+n}x}^2 \ge 0,
\end{equation}
for every $n \ge 0$,
since $x \in H$ is arbitrary, then
$T$ would be a $b$-contraction.
But \eqref{eq La_n an contr} follows immediately from the previous theorem,
since
$(1-\nabla)^a \La_n \ge 0$ for every $n \ge 0$ (because $T$ is an $a$-contraction)
and $(1-\nabla)^b \La$ is well defined (since $T \in \weak{b}$).
\end{proof}

We need to recall the
following asymptotic behavior of the Ces\`{a}ro numbers $k^a(n)$.

\begin{prop}\label{prop properties of Cesaro numbers}
If $a\in \C\setminus\{0,-1,-2,\dots\}$, then
\[
k^{a}(n)=\frac{\Gamma(n+a)}{\Gamma(a)\Gamma(n+1)}
={n+a-1\choose a-1}\qquad \forall n \ge 0,
\]
where $\Gamma$ is Euler's Gamma function. Therefore
\begin{equation}\label{asymp}
k^{a}(n)=\frac{n^{a-1}}{\Gamma(a)}(1+O(1/n)) \qquad \tn{ as } n\to\infty.
\end{equation}
Moreover, if $0< a \le 1$, then
\[
\frac{(n+1)^{a -1}}{\Gamma(a)}\leq k^{a}(n)\leq \frac{n^{a -1}}{\Gamma(a)} \qquad \forall n \ge 1.
\]
\end{prop}

\begin{proof}
See \cite[Vol. I, p. 77, equation (1.18)]{Zygmund} and \cite[eq. (1)]{ET}.
The last inequality follows
from the Gautschi inequality (see \cite[eq. (7)]{Gau59}).
\end{proof}

The next statement will be used in the proof of Theorem~\ref{thm consequence Q2}. 

\begin{thm}\label{thm ans to Q2}
Let $a>0$, and let the integer $m$ be defined by $m<a\le m+1$.
If $\La = \{ \La_n \}_{n \ge 0}$ is a sequence of real numbers,
then the following statements are equivalent:
\begin{enumerate}[\rm (i)]
\item
$(1-\nabla)^a \La \equiv 0$ (i.e., all the terms of the sequence $(1-\nabla)^a \La$ are $0$);
\item
$(1-\nabla)^{m+1} \La \equiv 0$;
\item
There exists a polynomial $p$ of degree at most $m$
such that $\La_n = p(n)$ for every $n \ge 0$.
\end{enumerate}
\end{thm}

\begin{proof}
The equivalence (ii) $\Leftrightarrow$ (iii) is a well known fact
(see \cite[Theorem~2.1]{BMN10}).
Suppose that (i) is true.
Let us see that
\begin{equation}\label{eq Kuttner eq 2 apl}
(1-\nabla)^{m+1} \La_n = (1-\nabla)^{m+1-a} [(1-\nabla)^{a} \La]_n
\end{equation}
for every $n \ge 0$.
Indeed, the RHS of \eqref{eq Kuttner eq 2 apl} is obviously $0$ by assumption.
Then we can apply Theorem~\ref{Kuttner Thm A} with $s = a > -1$ and $r=m-a \ge 0$,
and \eqref{eq Kuttner eq 2 apl} follows. Therefore we obtain that (i) $\Rightarrow$ (ii).

Suppose now that (ii) is true.
Hence, we also have (iii).
Obviously, if $a = m+1$ we obtain (i). 
Now let us prove (i) for $m < a < m+1$ (so $a$ is non-integer).
We will see that
\begin{equation}\label{eq Kuttner a and a-m-1 and m+1}
(1-\nabla)^a \La_n = (1-\nabla)^{a-m-1} [(1-\nabla)^{m+1} \La]_n
\end{equation}
for every $n \ge 0$. Indeed, fix $n \ge 0$. Then, by (iii) and \eqref{asymp}, we have that
\[
\La_{j+n} \lesssim_n (j+1)^{m}.
\]
Therefore
\[
\sum_{j=0}^{\infty} \abs{k^{-a}(j)} \La_{j+n} 
\lesssim_n 
\sum_{j=0}^{\infty} (j+1)^{-a-1} (j+1)^m = \sum_{j=0}^{\infty} (j+1)^{m-a-1} 
< \infty
\]
since $m<a$. Hence, the series
\[
\sum_{j=0}^{\infty} k^{-a}(j) \La_{j+n}
\]
converges for every $n \ge 0$. 
Thus the LHS of \eqref{eq Kuttner a and a-m-1 and m+1} is well defined. 
The RHS  of \eqref{eq Kuttner a and a-m-1 and m+1} 
is obviously well defined since we are assuming (ii).
Therefore, taking $s=m+1$ and $r+s = a$ in Theorem~\ref{Kuttner Thm A} 
we obtain that indeed \eqref{eq Kuttner a and a-m-1 and m+1} holds, 
and hence we obtain (i).
\end{proof}

\begin{proof}[Proof of Theorem~\ref{thm consequence Q2}]
The equivalence between statements (ii) and (iii) is well-known.
Suppose that (i) is true. Then, fixing $h \in H$
and taking $\La_n := \norm{T^n h}^2$, note that (ii) follows immediately
by applying Theorem~\ref{thm ans to Q2}.

Suppose now that we have (ii), that is, $T$ is an $(m+1)$-isometry.
Then, $\norm{T^n}^2 \lesssim (n+1)^{m}$. Therefore
\[
\sum_{n=0}^{\infty} (n+1)^{-a-1} \norm{T^n h}^2
\lesssim
\sum_{n=0}^{\infty} n^{m-a-1},
\]
for every $h \in H$. The last series above converges since $m<a$.
This means that $T \in \weak{a}$,
and then (i) follows using Theorem~\ref{thm ans to Q2} again.
\end{proof}

\begin{rem}
\label{rem before proof}
In \cite[Proposition 8]{Ath91}, for any positive integer $m$, 
Athavale gives an example of an operator $T$ 
(a unilateral weighted shift), 
which is an $(m+1)$-isometry but not $n$-isometry for any positive integer $n \le m$.
\end{rem}

We also have the following result.

\begin{thm}\label{thm bajar a-contr a b-contr con b entero}
Let $0<c<b<a$ where $c$ is not an integer. 
If $T$ is an $a$-contraction and $(1-t)^c$ belongs to $\CA_T$, 
then $T$ is a $b$-contraction.
\end{thm}

\begin{proof}
Fix $x \in H$ and let $\La_n := \norm{T^n x}^2$. Taking
\[
\si = c, \quad s = a, \quad r=b-a, \quad \tau =[a-b+c]+1
\]
in \eqref{eq Kuttner 1}, the statement follows.
\end{proof}

In fact, the only new case in this theorem is when $b$ is an integer, 
otherwise it follows from Theorem~\ref{thm consequence Q1}. 

Let us study now some properties relating weighted shifts with $a$-contractions.
Given $s>0$, recall that the space $\CH_s$ and the backward
shift $B_s$ on it have been defined in 
the Introduction, see~\eqref{eq defi space Ha}, \eqref{def-B-k} 
and a comment following these formulas.
The forward shift $F_s$ is defined on $\CH_s$ by
\begin{equation}
\label{defi Fs}
F_s f(t)= t f(t).
\end{equation}
The asymptotic behavior of the norms of the powers of $B_s$ and $F_s$ is
\begin{equation}\label{eq asymp Bs and Fs}
\norm{B_s^m}^2 \asymp (m+1)^{\max \{ 1-s, 0 \}}
\quad \tn{ and } \quad
\norm{F_s^m}^2 \asymp (m+1)^{\max \{ s-1, 0 \}}
\end{equation}
(see for example \cite[eq. (7.6)]{ABY19_1}).

\begin{thm}[{\cite[Theorem~7.2]{ABY19_1}}]
\label{thm B_s a-contraction}
Let $a$ and $s$ be positive numbers. Then:
\begin{itemize}
\item[\tn{(i)}] 
$B_s \in \weak{a}$; 
\item[\tn{(ii)}] 
$B_s$ is an $a$-contraction if and only if $a \le s$. 
\end{itemize}
\end{thm}

In order to obtain the corresponding result for the forward shift $F_s$,
we need to introduce the following sets:
\begin{equation}
\label{eq partition 0 infty into even and odd intervals}
J_{\tn{even}} := \bigcup_{j \in \Z_{\ge 0}} (2j,2j+1), 
\quad \quad 
J_{\tn{odd}} := \bigcup_{j \in \Z_{\ge 0}}
(2j+1,2j+2).
\end{equation}
Then $\{ J_{\tn{even}}, J_{\tn{odd}}, \N \}$ is a partition of the interval $(0, \infty)$.

\begin{thm}\label{thm Fs admiss}
Let $a$ and $s$ be positive numbers. 
Then $F_s \in \weak{a}$ if and only if $s<a+1$ or $a$ is integer.
\end{thm}

\begin{proof}
If $a$ is a positive integer, then obviously any operator in $L(H)$ belongs to $\weak{a}$,
since in this case $(1-t)^a$ is just a polynomial.
Suppose now that $a$ is not an integer.
We will use the notation of \cite[Theorem~2.15]{ABY19_1}.
Now $\ka(t) = (1-t)^{-s}$ and $\al(t) = (1-t)^a$. Therefore
\[
\be(\nabla) \ka_m
=
\sum_{n=0}^{\infty} \abs{k^{-a}(n)} \norm{F_s^n e_m} ^2
\asymp
\sum_{n=0}^{\infty} (n+1)^{-a-1} (n+m+1)^{s-1}.
\]
If $F_s \in \weak{a}$, then the above series must converge for every $m \ge 0$.
For each $m$, this series indeed behaves as $\sum n^{s-a-2}$.
Hence it implies that $s<a+1$.

Reciprocally, suppose now that $s<a+1$.
Then
\[
\begin{split}
\sum_{n=0}^{\infty} (n+1)^{-a-1} (n+m+1)^{s-1}
&\asymp
(m+1)^{s-1} \sum_{n=0}^{m} (n+1)^{-a-1} + \sum_{n=m+1}^{\infty} (n+1)^{s-a-2} \\
&\asymp
(m+1)^{s-1} + (m+1)^{s-a-1}
\asymp (m+1)^{s-1} \asymp \ka_m.
\end{split}
\]
Therefore \cite[Theorem~2.15~(i)]{ABY19_1} implies that
$F_s \in \weak{a}$.
\end{proof}

\begin{thm}\label{thm Fs a-contr}
Let $a$ and $s$ be positive numbers and $s<a+1$. Then:
\begin{itemize}
\item[\tn{(i)}]
$(1-t)^a(F_s^*, F_s)\ge 0$ (that is, $F_s$ is an $a$-contraction)
if and only if $s \in J_{\tn{even}} \cup \N$; 
\item[\tn{(ii)}]
$(1-t)^a(F_s^*, F_s)\le 0$ if and only if $s \in J_{\tn{odd}} \cup \N$; 
\item[\tn{(iii)}]
$(1-t)^a(F_s^*, F_s)=0$ (that is, $F_s$ is an $a$-isometry)
if and only if $s \in \N$.
\end{itemize}
\end{thm}

\begin{proof} By \cite[Theorem~2.15]{ABY19_1},
$(1-t)^a(F_s^*, F_s)\ge 0$ if and only if
$(1-\nabla)^a k^{s}(m)\ge 0$ for all $m\ge 0$, and similar assertions
hold in the context of (ii) and (iii). So we need to study the signs of
\[
(1-\nabla)^a k^{s}(m)=\sum_{n=0}^\infty k^{-a}(n)k^s(n+m).
\]
We assert that for any $s<a+1$,
\begin{equation}
\label{s01}
\begin{aligned}
(1-\nabla)^a k^{s}(m) = \frac{\sin(\pi s)\Gamma(1-s+a)\Gamma(s+m)}{\pi\Gamma(m+a+1)}.
\end{aligned}
\end{equation}

Suppose first that $a$ is a positive integer.
Then, by \cite[Example 3.4 (ii)]{AM18}, we have that
\[
(1-\nabla)^a k^{s}(m)=(-1)^{a}k^{s-a}(m+a),
\]
for every non-negative integer $m$,
and the statement follows easily.

Suppose that $s\in (0,1)$ and let $a$ be
any real number with $a>s-1$.
Using the expression for $k^{s}(m)$ given in
Proposition~\ref{prop properties of Cesaro numbers}
and applying
the idea of the proof of
\cite[Lemma~1.1]{AM18},
we have
\[
\begin{split}
(1-\nabla)^a k^{s}(m)
&=
\frac{1}{\Gamma(s)\Gamma(1-s)}
\sum_{l=0}^{\infty}k^{-a}(l)\frac{\Gamma(1-s)\Gamma(s+m+l)}{\Gamma(m+l+1)} \\
&=
\frac{1}{\Gamma(s)\Gamma(1-s)}
\sum_{l=0}^{\infty}k^{-a}(l)\int_{0}^1 x^{-s}(1-x)^{s+m+l-1}\,dx \\
&=
\frac{1}{\Gamma(s)\Gamma(1-s)}\int_{0}^1 x^{a-s}(1-x)^{s+m-1}\,dx
=
\frac{\Gamma(1-s+a)\Gamma(s+m)}{\Gamma(s)\Gamma(1-s)\Gamma(m+a+1)}.
\end{split}
\]
If $s=1$, it is immediate that $(1-\nabla)^a k^{s}(m)=0$.
This gives~\eqref{s01} for $s\in (0,1]$.

Next, assume that $s>1$.
The summation by parts formula gives
\[
\begin{split}
\sum_{n=0}^N k^{-a}(n)k^s(n+m)
&= k^s(m)+\sum_{n=1}^N (k^{-a+1}(n)-k^{-a+1}(n-1))k^s(n+m) \\
&= k^{-a+1}(N)k^s(N+m)+\sum_{n=0}^{N-1} (k^{s}(n+m)-k^{s}(n+m+1))k^{-a+1}(n)\\
&= k^{-a+1}(N)k^s(N+m)-\sum_{n=0}^{N-1} k^{s-1}(n+m+1)k^{-a+1}(n).
\end{split}
\]
By passing to the limit as $N\to\infty$ and using that $a>s-1$,
we obtain that
\[
(1-\nabla)^a k^{s}(m)= - (1-\nabla)^{a-1} k^{s-1}(m+1).
\]
This implies that whenever~\eqref{s01} holds for a pair $(a-1,s-1)$ (for all $m$),
it also holds for the pair $(a,s)$ and for all $m$.
Therefore, the case of an arbitrary pair $(a,s)$ reduces
to the case of the pair $(a-n,s-n)$, where $n<s\le n+1$, for which
~\eqref{s01} has been checked already. This proves this formula
for the general case. The sign of
$\sin(\pi s)$ depends on whether $s\in J_{\text{even}}$,
$s\in J_{\text{odd}}$, or $s$ is integer, whereas for $a,s>0$, all
values of $\Gamma$ in~\eqref{s01} are positive. This gives our statements.
\end{proof}

As an obvious consequence of Theorem~\ref{thm B_s a-contraction},
we obtain that it is not possible in general to pass from $b$-contractions
to $a$-contractions when $0<b<a$.

\begin{prop}
\label{thm-Cb-not-Ca}
Let $0<b<s<a$. Then $B_s$ is a $b$-contraction, but not an $a$-contraction.
\end{prop}

We also have the following result.

\begin{prop}
Let $0<a \le s<1$. Then $B_s$ is an $a$-contraction, which is not similar to a contraction.
\end{prop}

\begin{proof}
Since $a \le s$, Theorem~\ref{thm B_s a-contraction} gives that $B_s$ is an
$a$-contraction, and since $s<1$, \eqref{eq asymp Bs and Fs} gives that
$B_s$ is not power bounded.
\end{proof}

Moreover, passing from $a$-contractions to $b$-contractions (when $0<b<a$)
is neither possible in general, as the following statement shows.
It is an immediate consequence of Theorems~\ref{thm Fs admiss} and \ref{thm Fs a-contr}. 

\begin{prop}
Let $1< a \le 2$ and $0<b<a$. If $\max\{ 2,b+1 \} < s < a+1$,
then $F_s$ is an $a$-contraction, but does not belong to $\weak{b}$ 
(so in particular, $F_s$ is not a $b$-contraction).
\end{prop}

\begin{prop}
Let $0< s <1$. If $0 < a \le \min \{ s, 1-s \}$ then $B_s$ is an $a$-contraction, 
but the series $\sum k^{-a}(n) \norm{B_s^n}^2$ does not converge. 
\end{prop}

\begin{proof}
Since $a \le s$, 
Theorem~\ref{thm B_s a-contraction} gives that $B_s$ is an $a$-contraction. 
Moreover, using that $a \le 1-s$ and \eqref{eq asymp Bs and Fs}, 
it is immediate that
\[
\sum_{n=0}^{\infty} |k^{-a}(n)| \norm{B_s^n}^2 = \infty,
\]
and the statement follows.
\end{proof}

\begin{thm}\label{thm even contr to odd contr and vice versa}
Let $m$ be a positive integer.
\begin{enumerate}[\rm (i)]
\item
If $T$ is a $(2m+1)$-contraction, then $T$ is a $2m$-contraction and
\[
\norm{T^n x}^2 \lesssim (n+1)^{2m} \quad \quad (\forall x \in H).
\]
\item
If $T$ is a $(2m)$-contraction and
\[
\norm{T^n x}^2 = o(n^{2m-1}) \quad \quad (\forall x \in H),
\]
then $T$ is a $(2m-1)$-contraction.
\end{enumerate}
\end{thm}

\begin{rems}
\quad
\begin{enumerate}
\item[(a)]
The fact that $(2m+1)$-contractions are $2m$-contractions was already proved
by Gu in \cite[Theorem 2.5]{Gu15}. Here we give an alternative proof, which also works for (ii).
\item[(b)]
In general, $2m$-contractions are not $(2m-1)$-contractions.
For example,
the forward weighted shift $F_2$ is a $2$-isometry (see Theorem~\ref{thm Fs a-contr} (iii)),
but it is not a contraction. Indeed, it is not power bounded (see \eqref{eq asymp Bs and Fs}).
\end{enumerate}
\end{rems}

\begin{proof}[Proof of Theorem~\ref{thm even contr to odd contr and vice versa}]
Fix $x\in H$ and put $\La_n := \norm{T^n x}^2$, for every $n \ge 0$.
It is easy to see (for instance, by induction on $k$) that
\begin{equation}
\label{eq Taylor finite differ}
\begin{split}
\La_n
&=
\sum_{j=0}^{k-1} (-1)^j {n \choose j} (I-\nabla)^j \La_0 \,
+ \, (-1)^k \sum_{j=0}^{n-k} {n-1-j \choose k-1}  (I-\nabla)^k \La_j \\
&=:
(I) + (II),
\end{split}
\end{equation}
for $n \ge k$.
(Note that this formula is a discrete analogue of the Taylor formula with
the rest in the integral form.)

Let us prove (i).  Suppose that $T$ is a $(2m+1)$-contraction.
Taking $k = 2m+1$ (and $n$ sufficiently large) in \eqref{eq Taylor finite differ},
we obtain that $(II) \le 0$, since $(I-\nabla)^{2m+1} \La_j \ge 0$. Therefore
\begin{equation}
\label{eq Taylor finite differ odd}
\La_n \le \La_0 - {n \choose 1} (I-\nabla) \La_0 + \cdots + {n \choose 2m} (I-\nabla)^{2m} \La_0.
\end{equation}
If $(I-\nabla)^{2m} \La_0 < 0$, then the RHS of \eqref{eq Taylor finite differ odd}
is a polynomial in $n$ of degree $2m$ whose main coefficient is negative.
Hence, it is negative for $n$ sufficiently large.
This contradicts the fact that $\La_n \ge 0$ for every $n$.
Therefore, $(I-\nabla)^{2m} \La_0 \ge 0$. Since the vector $x \in H$, fixed at the beginning
of the proof, was arbitrary, this means that $T$ is a $2m$-contraction.
We have also obtained that $\La_n \lesssim (n+1)^{2m}$.
This completes the proof of (i).

Assume the hypotheses of (ii).
Now taking $k = 2m$ (and $n$ sufficiently large) in \eqref{eq Taylor finite differ},
we obtain that $(II) \ge 0$, since $(I-\nabla)^{2m} \La_j \ge 0$. Therefore
\begin{equation}
\label{eq Taylor finite differ even}
\La_n \ge \La_0 - {n \choose 1} (I-\nabla) \La_0 
+ \cdots - {n \choose 2m-1} (I-\nabla)^{2m-1} \La_0.
\end{equation}
If $(I-\nabla)^{2m-1} \La_0 < 0$, then the RHS of \eqref{eq Taylor finite differ odd}
is a polynomial in $n$ of degree $2m-1$ whose main coefficient is positive.
But this contradicts the hypothesis $\La_n = o(n^{2m-1})$, hence it must be
$(I-\nabla)^{2m-1} \La_0 \ge 0$.
Since the vector $x \in H$, fixed at the beginning
of the proof, was arbitrary, this means that $T$ is a $(2m-1)$-contraction.
\end{proof}


\section{Remarks on the models for $a$-contractions with $a>1$}
\label{Section a bigger than 1}

In Theorem~3.51 of his thesis \cite{Schillo-tesis}, 
Schillo proves that a commutative 
operator tuple belongs to certain classes if and only if it can be modeled 
as a compression of the tuple of multiplication operators 
by coordinates on some natural Bergman-type spaces of the unit ball.  
Specialized to the case of one operator, this result implies 
that for $a \ge 1$, the following statements are equivalent: 
\begin{enumerate}[\rm (1)]
\item
$T$ is a contraction and an $a$-contraction;
\item 
there exists a separable Hilbert space $\CE$ such that
$T$ is unitarily equivalent to a part of an operator $(B_a \otimes I_\CE)\oplus U$,
where $U$ is an unitary operator. 
\end{enumerate}

Here we discuss the models of $a$-contractions, without extra assumptions. 
In view of Theorem~\ref{thm model a in 01},
which gives a model for $a$-contractions when $0<a<1$,
it is natural to ask whether for $a>1$, 
the statements
\begin{enumerate}[\rm (a)]
\item
$T$ is an $a$-contraction,
\item
there exists a separable Hilbert space $\CE$ such that
$T$ is unitarily equivalent to a part of an operator $(B_a \otimes I_\CE)\oplus S$,
where $S$ is an $m$-isometry,
\end{enumerate}
are equivalent (here $m$ is the integer defined by $m-1 < a \le m$).

It turns out that one implication is true, but the other is false in general.

\begin{thm}
\label{thm model a>1 true implication}
Let $a>1$ and let $m$ be the positive integer such that $m-1 < a \le m$.
Then any part of $(B_a \otimes I_\CE)\oplus S$, where
$S$ is an $m$-isometry and $\CE$ is an auxiliary Hilbert space, is
an $a$-contraction.
\end{thm}

For the proof of this theorem we reproduce the 
following lemma from \cite{ABY19_1}.

\begin{lemma}[{\cite[Proposition 2.6]{ABY19_1}}]
\label{lemma direct sum a-contr}
Let $a>0$.
\begin{enumerate}[\rm (i)]
\item
If $T$ is an $a$-contraction, then any part of $T$ is also an $a$-contraction.
\item
If $T_1$ and $T_2$ are $a$-contractions,
then $T_1 \oplus T_2$ is also an $a$-contraction.
\item
If $T$ is an $a$-contraction, then $T \otimes I_\CE$
(where $I_\CE$ is the identity operator on some Hilbert space $\CE$)
is also an $a$-contraction.
\end{enumerate}
\end{lemma}

\begin{proof}[Proof of Theorem~\ref{thm model a>1 true implication}]
By Theorem~\ref{thm B_s a-contraction}, $B_a$ is an $a$-contraction,
and, by Theorem~\ref{thm consequence Q2},
$S$ is also an $a$-contraction.
Therefore the previous lemma implies that
$(B_a \otimes I_\CE)\oplus S$ is an $a$-contraction.
\end{proof}

\begin{prop}
\label{prop a>1 counterexample}
Let $a$ belong to the set
\[
A:= \bigcup_{j\ge 1} \, (2j-1,2j] \ss \R.
\]
Let $m$ be the positive integer such that $m-1 < a \le m$,
and take $s \in (m,a+1)$ (such $s$ exists).
Then the forward weighted shift $F_s$ is an $a$-contraction that
cannot be modeled by a part of $(B_a \otimes I_\CE)\oplus S$,
where $S$ is an $m$-isometry.
\end{prop}

\begin{proof}
Assume all the hypothesis of the statement.
Since $s<a+1$ and $s \in J_{\tn{even}}$, Theorem~\ref{thm Fs a-contr}~(i)
gives that $F_s$ is an $a$-contraction.
The second part of the statement follows by comparison of
operator norms.
By \eqref{eq asymp Bs and Fs},
$\norm{F_s^n}^2 \asymp (n+1)^{s-1}$, for every $n\ge 0$.
On the other hand, $B_a$ is a contraction,
and $\norm{S^n}^2 \lesssim (n+1)^{m-1}$
since $S$ is an $m$-isometry.
(This last asymptotics is well-known. For instance, it follows immediately
from \cite[Theorem~2.1]{BMN10}.)
Therefore
\[
\norm{((B_a \otimes I_\CE)\oplus S)^n}^2
\lesssim
(n+1)^{m-1}.
\]
Since $m-1 < s-1$, we get that
$F_s$ cannot be modeled by a part of $(B_a \otimes I_\CE)\oplus S$. 
\end{proof}

\begin{rem}
It remains open whether (a) implies (b) when $a>1$ belongs to $\R \sm A$.
\end{rem}


\section*{Acknowledgments}
The first author has been partly supported by Project MTM2016-77710-P,
DGI-FEDER, of the MCYTS, Project E26-17R, D.G. Arag\'on, 
and Project for Young Researchers, Fundación Ibercaja and Universidad de Zaragoza, Spain.
The second author has been partially supported by
La Caixa-Severo Ochoa grant
(ICMAT Severo Ochoa project SEV-2011-0087, MINECO).
Both second and third authors acknowledge partial support by
Spanish Ministry of Science, Innovation and
Universities (grant no. PGC2018-099124-B-I00) and
the ICMAT Severo Ochoa project SEV-2015-0554 of the Spanish Ministry of Economy and
Competitiveness of Spain and the European Regional Development
Fund, through the ``Severo Ochoa Programme for Centres of Excellence
in R$\&$D''.

\bibliographystyle{siam}
\bibliography{biblio_ABY}

\end{document}